\documentclass[12pt]{amsart}
\usepackage{pgf,tikz,pgfplots} 
\pgfplotsset{compat=1.14}
\usepackage{mathrsfs}
\usetikzlibrary{arrows}
\usetikzlibrary{patterns}
\usepackage{pgfplots}
\usetikzlibrary{intersections, pgfplots.fillbetween}
\usepackage[colorlinks=true,urlcolor=blue, citecolor=red,linkcolor=blue,linktocpage,pdfpagelabels, bookmarksnumbered,bookmarksopen]{hyperref}
\usepackage[hyperpageref]{backref}
\usepackage{cleveref}
\usepackage{xcolor}
\usepackage{amsthm} 
\usepackage{latexsym,amsmath,amssymb}
\usepackage{accents}
\usepackage[colorinlistoftodos,prependcaption,textsize=tiny]{todonotes}
\usepackage{a4wide}
\usepackage{soul}
\usepackage{mathtools} 
\usepackage{xparse} 

\pgfdeclarelayer{ft}
\pgfdeclarelayer{bg}
\pgfsetlayers{bg,main,ft}

\title{On minimizing $W^{s,1/s}$-maps between circles}

\date{\today}
\author{Dorian Martino}
\address[Dorian Martino]{
	ETH Zürich,%
	Department of Mathematics,
	Rämistrasse 101,
	8092 Zürich, Switzerland}
\email{dorian.martino@math.ethz.ch}
\author{Katarzyna Mazowiecka}
 \address[Katarzyna Mazowiecka]{
	 Institute of Mathematics,%
	 University of Warsaw,
	 Banacha 2,
	 02-097 Warszawa, Poland}
 \email{k.mazowiecka@mimuw.edu.pl}
\author{Armin Schikorra}
 \address[Armin Schikorra]{Department of Mathematics,
	 University of Pittsburgh,
	 301 Thackeray Hall,
	 Pittsburgh, PA 15260, USA}
 \email{armin@pitt.edu}

\definecolor{indigo}{rgb}{0.29, 0.0, 0.51}
\definecolor{p1}{gray}{0.4}
\definecolor{p2}{gray}{0.6}
\definecolor{p3}{gray}{0.98}
\definecolor{p4}{gray}{0.8}
\definecolor{p5}{gray}{0.9}

plus 6pt minus 12pt
	\renewcommand{\i}{{\rm \bf i}}
	\def\eps{\varepsilon}

	\def\id{{\rm id }}

	\def\C{{\mathbb C}}

	\newcommand{\dif}{\,\mathrm{d}}
	
	\def\N{{\mathbb N}}

	\def\S{{\mathbb S}}

	\newtheorem{theorem}{Theorem}
	\newtheorem{lemma}[theorem]{Lemma}
	\newtheorem{corollary}[theorem]{Corollary}
	\newtheorem{proposition}[theorem]{Proposition}

	\newtheorem{definition}[theorem]{Definition}

	\def\lip{{\rm Lip\,}}

	\newcommand{\dd}{\,\mathrm{d}}
	\newcommand{\dx}{\dif x}
	\newcommand{\dy}{\dif y}

	\newcommand{\R}{\mathbb{R}}

	\newcommand{\Z}{\mathbb{Z}}

	\newcommand{\brac}[1]{\left (#1 \right )}
	\newcommand{\abs}[1]{\left |#1 \right |}

	\newcommand{\barint}{
		\rule[.036in]{.12in}{.009in}\kern-.16in \displaystyle\int }
	
	\newcommand{\barcal}{\mbox{$ \rule[.036in]{.11in}{.007in}\kern-.128in\int $}}
	
	\def\mvint_#1{\mathchoice
		{\mathop{\vrule width 6pt height 3 pt depth -2.5pt
				\kern -8pt \intop}\nolimits_{\kern -3pt #1}}%
	{\mathop{\vrule width 5pt height 3 pt depth -2.6pt
			\kern -6pt \intop}\nolimits_{#1}}%
	{\mathop{\vrule width 5pt height 3 pt depth -2.6pt
			\kern -6pt \intop}\nolimits_{#1}}%
	{\mathop{\vrule width 5pt height 3 pt depth -2.6pt
			\kern -6pt \intop}\nolimits_{#1}}}

\numberwithin{theorem}{section} \numberwithin{equation}{section}

\newcommand{\lap}{\Delta }
\newcommand{\aleq}{\precsim}

\let\latexchi\chi
\makeatletter
\renewcommand\chi{\@ifnextchar_\sub@chi\latexchi}
\newcommand{\sub@chi}[2]{
	\@ifnextchar^{\subsup@chi{#2}}{\latexchi^{}_{#2}}%
}
\newcommand{\subsup@chi}[3]{
	\latexchi_{#1}^{#3}%
}
\makeatother

\begin{document}
	
	\begin{abstract}
	For $s \in (\frac{1}{4},1)$ and any degree the only $W^{s,\frac{1}{s}}$-minimizers for $\S^1 \to \S^1$ maps are Blaschke products. This gives a resolution of Open Problems 23 and 24 in \cite{BM-book}, Brezis' favorite Open Problem 5.4 \cite{BFavorite} in this $s$-range. Previous results of this type were partial and restricted only to a small neighborhood of $s=\frac{1}{2}$.

In particular, we completely settle Open Problem 5.1 of Brezis' favorite problems, \cite{BFavorite}.
Moreover, as a consequence of the argument one also obtains linearized stability results.
\end{abstract}
	
	\sloppy
	
	\maketitle
	\tableofcontents
	\sloppy

	\section{Introduction}
Let $1<p<\infty$ and let $s=\frac1p$. For maps $u\colon \S^1\to\S^1$, we define
\[
E_s(u)
\coloneqq 
\int_{\S^1}\int_{\S^1}
\frac{|u(x)-u(y)|^p}{|x-y|^2}\,\dx\,\dy.
\]
Here $\S^1$ is identified with the unit circle in $\C$, and $\dx,\dy$
denote arclength measure. For any $s \in (0,1)$ this energy is conformally invariant under domain M\"obius transformations. We are interested in minimizers of $E_s$ under the assumption that $u$ has a prescribed degree. Despite recent progress and new methods developed, for the history see, e.g., \cite{MP04,BMRS14,Sucks23,MS23,DSW23,MMS26}, the following questions were open except for $s \approx \frac{1}{2}$, cf. \cite[Open Problem 23, 24]{BM-book}:

\begin{itemize}
 \item Does there exist a minimizer for all degrees $d\in\mathbb Z$?
 \item What is the minimum energy for each degree $d\in\mathbb Z$?
\item What are the minimizers for each degree?
 \end{itemize}

It is very natural to believe that for $s \in (0,1)$ the maps $z^d\colon \S^1 \to \S^1$ are the minimizers of degree $d$. But quite surprisingly, that is wrong.
\begin{theorem}[\cite{MMS26}]
For $s \in (0,\frac{1}{8})$, the identity is not a degree-one minimizer of $E_s$.
\end{theorem}
In \cite{MMS26} a stability machinery was developed that in particular showed that for $s \approx \frac{1}{2}$ all degree one minimizers are M\"obius maps.
Here, we settle the above three questions in the range $(\frac{1}{4},1)$.
Before we state our main theorem, we define the Blaschke products which serve as natural candidate for minimizers of degree $d$, since the energy is conformally invariant.

\begin{definition}[Blaschke products]
For $d\ge1$, let $\mathscr{B}_d^+$ denote the family of degree $d$ finite
Blaschke products:
\[
\mathscr{B}_d^+
=
\left\{
e^{\i\alpha}
\prod_{j=1}^d
\frac{z-a_j}{1-\overline{a_j}z}
:
\alpha\in\R,\ a_1,\ldots,a_d\in\mathbb{D}
\right\}.
\]
For $d\le -1$, define
\[
\mathscr{B}_d^-
=
\left\{
\overline{B}\colon B\in \mathscr{B}_{|d|}^+
\right\}.
\]
\end{definition}

The following is our main result: all minimizers are Blaschke Products of their respective degree.

\begin{theorem}\label{th:blaschkemin}
Let $1<p\le4$, $s=\frac1p$, and let $d \in \Z$.
If $u \in W^{s,\frac{1}{s}}(\S^1,\S^1)$ has degree $d$, then
\[
E_s(u)\ge |d|E_s(\id) = |d| 2^p\pi^{3/2}
\frac{
\Gamma\left(\frac{p-1}{2}\right)
}{
\Gamma\left(\frac p2\right)
}.
\]
Moreover, equality holds if and only if the following alternatives hold:
\begin{itemize}
 \item If $d > 0$ then $u\in \mathscr{B}_d^+$
 \item If $d < 0$ then $u\in \mathscr{B}_d^-$
 \item if $d=0$ then $u$ is constant.
\end{itemize}
\end{theorem}

\Cref{th:blaschkemin} in particular solves Brezis' favorite problem 5.1, cf. \cite[Open Problem 5.1]{BFavorite} and also \cite[Remark 7]{BrezisNewQuestions}. Indeed, we have by \Cref{la:brezis51} and \Cref{th:blaschkemin}
\begin{equation}\label{eq:brezis51}
|\deg u| \leq \frac{p-1}{4\pi}\, E_{1/p}(u), \qquad \forall p \in (1,2).
\end{equation}

\begin{corollary}
For any $s \in (\frac{1}{4},1)$ and any degree $d \in \Z$ the energy $E_s$ attains its minimum at exactly $z^d$ modulo Blaschke products.
\end{corollary}

Essentially as a corollary we also obtain local and global stability results in the $\dot{H}^{\frac{1}{2}}$-norm, i.e.\ the analogue of \cite[Theorem 1.3 and 1.5]{DSW23}.
\begin{theorem}
\label{s:stability}
Let $s\in(\frac12,1)$ and set $p=\frac1s\in(1,2)$. Let
$u\in W^{s,\frac1s}(\S^1,\S^1)$.
\begin{enumerate}
\item If $\deg u=1$, then there exists a constant $C_p>0$ such that
\[
\inf_{B_0\in{\rm B}_1^+}
\|u-B_0\|_{\dot H^{\frac{1}{2}}(\S^1)}^2
\le
C_p\, D(u),
\]
where $D(u)$ is the defect measure defined in \Cref{def:defect}.

\item If $\deg u=d\ge2$, then there exist constants
$\eps_{d,p}>0$ and $C_{d,p}>0$ such that, whenever
\[
\inf_{B_0\in{\rm B}_d^+}
\|u-B_0\|_{\dot H^{\frac{1}{2}}}^2
< \eps_{d,p},
\]
one has
\[
\inf_{B_0\in{\rm B}_d^+}
\|u-B_0\|_{\dot H^{\frac{1}{2}}}^2
\le
C_{d,p}\, D(u).
\]
\end{enumerate}
\end{theorem}

\subsection*{Outline}
In \Cref{s:preliminaries} we collect some preliminary formulae.

The main idea is to compare the energy $E_s$ with the energy $E_{1/2}$.

The $L^2$-energy $E_{1/2}$ admits an (already known) explicit decomposition into a degree term and a defect measure $D(u)$, see \Cref{la:E12vsdefect} and \Cref{def:defect}.
The main novelty is that --- despite the non-Hilbertian structure of $W^{s,\frac{1}{s}}(\S^1,\R^2)$ --- we obtain the explicit relation
\[
E_s(u)-\, E_s(z^{\deg u})
=
4\pi^2
\sum_{n=1}^{\infty}\gamma_{\frac{1}{s},n}D(u^n).
\]
where $D(\cdot)$ is the old $H^{1/2}$-defect and $\gamma_{p,n}$ are computable constants, see \Cref{pr:Fourierexpansion} and \Cref{pr:Esvsid}.
This elegant and surprising formula provides the main mechanism used in
\Cref{s:proofs} to prove \Cref{th:blaschkemin} and \Cref{s:stability} -- using the regularity theory we discuss in \Cref{s:regularity}.
\subsection*{Acknowledgement}

Part of this work was carried out while K.M. and A.S. were visiting University of Bielefeld. Further mutual visits between the author's institutions are gratefully acknowledged. A.S. was an Alexander-von-Humboldt Fellow. A.S. is funded by NSF Career DMS-2044898. D.M. is funded by Swiss National Science Foundation, project SNF $200020\textunderscore 219429$.

The project is co-financed by the Polish National Agency for Academic Exchange within Polish Returns Programme - BPN/PPO/2021/1/00019/U/00001 (K.M.). The project is co-financed by National Science Centre grant 2022/01/1/ST1/00021 (K.M.).

\textbf{Usage of LLM:} The authors used chatgpt to assist with conceptualization and computations. All mathematical validation, final proof decisions, and final wording remain the sole responsibility of the human authors.

\subsection*{Notation}We shall use the notation $A\aleq B$ to indicate that there exists a constant $C>0$, independent of the relevant parameters, such that $A\leq C B$. We write $A\approx B$ if both $A\aleq B$ and $B\aleq A$ hold.

\section{Preliminaries}\label{s:preliminaries}

We begin our preliminary section by recalling well-known energy formula
\begin{lemma}[{\cite{MMS26}}]\label{la:energofzk}
	It holds
	\[
	E_\frac{1}{2}(z^k) = 4\pi^2 |k|.
	\]
	More generally, we have
	\[
	E_{s}(z^k) = |k| E_s(z).
	\]
\end{lemma}

\begin{lemma}[{\cite[Chapter 5]{NIST}, Basic facts about the Gamma function}]\label{la:GammaFacts}
	For $x>0$, the Gamma function is defined by
	\[
	\Gamma(x)\coloneqq \int_0^\infty t^{x-1}e^{-t}\,\dd t.
	\]
	It satisfies
	\[
	\Gamma(x+1)=x\Gamma(x)
	\qquad\text{for every }x>0.
	\]
	In particular for $n \in \N$
	\begin{equation}\label{eq:prodformula}
		\prod_{j=1}^n (z+j-1)
		=
		\frac{\Gamma(z+n)}{\Gamma(z)},
		\qquad
		\prod_{j=1}^n (z-j)
		=
		\frac{\Gamma(z)}{\Gamma(z-n)}.
	\end{equation}
	\begin{equation}\label{eq:Gamma12}
		\Gamma(1/2) = \sqrt{\pi}
	\end{equation}
	
	It extends meromorphically to $\mathbb C \setminus \{0,-1,-2,\ldots\}$ and $(x>0\mapsto \log\Gamma(x))$ is convex.
	
	We shall use the following two standard identities. First, the relation between $B$ and $\Gamma$-function
	\begin{equation}\label{eq:betaidentity}
		B(a,b)\coloneqq \int_0^1 t^{a-1}(1-t)^{b-1}\,\dd t
		=
		\frac{\Gamma(a)\Gamma(b)}{\Gamma(a+b)}
		\qquad
		\text{for }a,b>0.
	\end{equation}
	Second, the Legendre relation:
	\begin{equation}\label{eq:legendrerel}
		\Gamma(z)\Gamma\left(z+\frac12\right)
		=
		2^{1-2z}\sqrt{\pi}\,\Gamma(2z).
	\end{equation}
	And third, Euler's reflection formula
	\begin{equation}\label{eq:reflectionformula}
		\Gamma(z)\Gamma(1-z)
		=
		\frac{\pi}{\sin(\pi z)} \quad z \not \in \Z.
	\end{equation}
\end{lemma}

The following identity is also well-known.
\begin{lemma}[{\cite[formula 5.12.5]{NIST}, A trigonometric Gamma integral}]\label{la:TrigGammaIntegral}
	Let $a>-1$ and $n\ge0$. Then
	\begin{equation}\label{eq:trigintegral}
		\int_0^\pi \sin^a (x)\,\cos(2nx)\,\dx
		=
		(-1)^n
		\frac{
			\pi\Gamma(a+1)
		}{
			2^a
			\Gamma\left(1+\frac a2+n\right)
			\Gamma\left(1+\frac a2-n\right)
		}.
	\end{equation}
	If the denominator contains a pole of the Gamma function, equivalently if
	$\frac a2\in\{0,1,\ldots,n-1\}$, the right-hand side is interpreted as
	$0$.
\end{lemma}

\begin{proof}
	The formula \eqref{eq:trigintegral} follows (after a change of variables) from the well-known identity
	\[
	\int_0^{\pi/2}(\cos t)^{A-1}\cos(bt)\,\dif t
	=
	\frac{\pi\Gamma(A)}
	{2^A
		\Gamma\left(\frac{A+b+1}{2}\right)
		\Gamma\left(\frac{A-b+1}{2}\right)},
	\qquad A>0,
	\]
	see, e.g., \cite[formula 5.12.5]{NIST}. 
	
	Since $a>-1$, the only possible pole in the denominator is in
	$\Gamma(1+\frac a2-n)$. This occurs precisely when
	\[
	1+\frac a2-n\in \{0,-1,-2,\ldots\},
	\]
	equivalently when $\frac a2\in\{0,1,\ldots,n-1\}$. In these cases the
	right-hand side is understood through the reciprocal Gamma function, so that
	$1/\Gamma(-k)=0$.
\end{proof}

The following computation, combined with \Cref{th:blaschkemin} implies \eqref{eq:brezis51}
\begin{lemma}\label{la:brezis51}
For $1<p<2$ set
\[
F(p)\coloneqq 
\frac{\Gamma\left(\frac p2\right)}
{\Gamma\left(\frac{p-1}{2}\right)}
2^{-p}\pi^{-3/2}\frac{1}{p-1}.
\]
Then, it holds
\[
\sup_{1<p<2} F(p)=\frac{1}{4\pi}.
\]
\end{lemma}
\begin{proof}
Set
\[
x\coloneqq \frac{p-1}{2}\in\left(0,\frac12\right).
\]
Then we have
\[
F(p)
=
\frac{\Gamma\left(x+\frac12\right)}{\Gamma(x)}
2^{-2x-1}\pi^{-3/2}\frac{1}{2x}
=
\frac{\Gamma\left(x+\frac12\right)}{\Gamma(x+1)}
2^{-2x-2}\pi^{-3/2}.
\]
Its log derivative in $x$ is given by 
\[
\frac{\dd}{\dd x}\log F
=
\psi\left(x+\frac12\right)-\psi(x+1)-2\log 2,
\]
where $\psi=\Gamma'/\Gamma$. Since $\psi$ is strictly increasing on $(0,\infty)$ (because $\log \Gamma$ is convex on the real axis), we have
\[
\psi\left(x+\frac12\right)-\psi(x+1)<0.
\]
Thus $F$ is strictly decreasing on $(1,2)$. Therefore
\[
\sup_{1<p<2} F(p)
=
\lim_{p\downarrow 1} F(p).
\]
Using $\Gamma(\varepsilon)\sim \varepsilon^{-1}$ as $\varepsilon\downarrow0$, we obtain 
\[
\Gamma\left(\frac{p-1}{2}\right)
\sim \frac{2}{p-1}.
\]
Hence, it holds
\[
\lim_{p\downarrow1}
\frac{\Gamma\left(\frac p2\right)}
{\Gamma\left(\frac{p-1}{2}\right)}
\frac{1}{p-1}
=
\frac{\Gamma\left(\frac12\right)}{2}
=
\frac{\sqrt\pi}{2}.
\]
Since $2^{-p}\to 2^{-1}$, we obtain
\[
\lim_{p\downarrow1}F(p)
=
\frac{\sqrt\pi}{2}\cdot \frac12\cdot \pi^{-3/2}
=
\frac{1}{4\pi}.
\]
\end{proof}

\subsection{Defect, quadratic energy and degree}
We recall the Cauchy index formula for the degree for maps $v\colon \S^1 \to \S^1$ (see, e.g., \cite[Equation  (12.1)]{BM-book})
\begin{equation}\label{eq:Cauchyindex}
\begin{split}
 \deg v = \frac{1}{2\pi \i} \int_0^{2\pi} \frac{v'(e^{\i t} )}{v(e^{\i t})} \dd t
 =
\frac1{2\pi \i}
\int_0^{2\pi}
\overline{v(e^{\i t})}\frac{\dd}{\dd t}v(e^{\i t})\,\dd t.
\end{split}
\end{equation}

\begin{lemma}\label{la:degun}
If $u\colon  \S^1 \to \S^1$ then if $\deg u \geq 0$ and $n \in \{1,2,\ldots,\}$ then
\[
 \deg(u^n) = n \deg(u).
\]
\end{lemma}
\begin{proof}
This is a direct computation using \eqref{eq:Cauchyindex}:
\[
\begin{split}
 \deg u^n
 &=\frac1{2\pi i}
\int_0^{2\pi}
\overline{u^n(e^{it})}\frac{\dd}{\dd t}\brac{u^n(e^{it})}\dd t\\[2mm]
 &=n \frac1{2\pi i}
\int_0^{2\pi}
\overline{u^n(e^{it})} u^{n-1}(e^{it})\frac{\dd}{\dd t}\brac{u(e^{it})}\dd t\\[2mm]
 &=n \frac1{2\pi i}
\int_0^{2\pi}
\overline{u^n(e^{it})} \frac{1}{\overline{u^{n-1}(e^{it}})}\frac{\dd}{\dd t}\brac{u(e^{it})}\dd t\\[2mm]
 &=n \frac1{2\pi i}
\int_0^{2\pi}
\overline{u(e^{it})} \frac{\dd}{\dd t}\brac{u(e^{it})}\dd t\\[2mm]
&=n \deg u.
\end{split}
\]
\end{proof}

\begin{lemma}\label{la:holomextvsdeg}
Let $v\colon \S^1 \to \S^1$ be continuous and let $V\colon  \mathbb{D} \to \C$ be any extension of $v$ to the whole disc $\mathbb{D}$. Then
\begin{equation}\label{eq:holomextvsdeg}
 \int_{\mathbb{D}}|\partial_{\bar{z}} V|^2 = \frac{1}{4} \int_{\mathbb{D}} |\nabla V|^2 - \frac{\pi}{2} \deg v,
\end{equation}
where $\partial_{\bar{z}} \coloneqq  \frac{1}{2}  \brac{\partial_{x} + \i \partial_y}$ is the usual Wirtinger derivative.
\end{lemma}
\begin{proof}
Since
\[
\begin{split}
 &4\, \partial_{\bar{z}} V \overline{\partial_{\bar{z}} V}\\[2mm]
 &= 4\, \partial_{\bar{z}} V \partial_{z} \overline{V}\\[2mm]
 &=  \brac{\partial_{x}+\i \partial_y} \brac{\Re(V) + \i \Im (V)} \brac{\partial_{x}-\i \partial_y} \brac{\Re(V) - \i \Im (V)}\\[2mm]
 &=  \brac{\brac{\partial_{x}+\i \partial_y}\Re(V) + \brac{\i \partial_{x}- \partial_y} \Im (V)}  \brac{\brac{\partial_{x}-\i \partial_y}\Re(V) -  \brac{\i \partial_{x}+ \partial_y}\Im (V)}\\[2mm]
&= \brac{\partial_{x}\Re(V) - \partial_y \Im (V) +\i \partial_y \Re(V) + \i \partial_{x}\Im (V)}  \brac{\partial_{x}\Re(V)- \partial_y\Im (V)-\i \partial_y\Re(V) -  \i \partial_{x}\Im (V)}\\[2mm]
&=  \brac{\partial_{x}\Re(V) - \partial_y \Im (V) +\i \partial_y \Re(V) + \i \partial_{x}\Im (V)}  \brac{\partial_{x}\Re(V)- \partial_y\Im (V)-\i \brac{\partial_y\Re(V) +  \partial_{x}\Im (V)}}\\[2mm]
&= \brac{\partial_{x}\Re(V) - \partial_y \Im (V)}^2 + \brac{\partial_y \Re(V) + \partial_{x}\Im (V)}^2\\[2mm]
&= \brac{\partial_{x}\Re(V)}^2 +  \brac{\partial_y \Im (V)}^2 - 2 \partial_{x}\Re(V)\, \partial_y \Im (V) + \brac{\partial_y \Re(V)}^2 + \brac{\partial_{x}\Im (V)}^2 + 2 \partial_y \Re(V)\, \partial_{x}\Im (V)\\[2mm]
&= |\nabla V|^2 -2 \det (\nabla \Re{V}, \nabla \Im{V}).
\end{split}
 \]
We have \eqref{eq:holomextvsdeg} by integration.
\end{proof}

\begin{lemma}[Fourier expansion of the quadratic energy]\label{la:E12vsdefect}
Let $v\colon \S^1\to\S^1$ be a smooth map given as
\[
v(e^{\i t})=\sum_{k\in\Z}a_k e^{\i kt}.
\]
Then
\begin{equation}\label{eq:1/2energyfourier}
E_{1/2}(v)
=
4\pi^2
\sum_{k\in\Z}|k|\,|a_k|^2,
\end{equation}
and
\begin{equation}\label{eq:degreefourier}
\deg v
=
\sum_{k\in\Z}k\,|a_k|^2.
\end{equation}
Therefore, if $\deg v\ge0$, then
\[
E_{1/2}(v)
=
4\pi^2\deg v
+
8\pi^2
\sum_{k<0}|k|\,|a_k|^2.
\]
\end{lemma}

\begin{proof}
The proof of \eqref{eq:1/2energyfourier} can be found in \cite[Lemma 12.5]{BM-book} and the proof of \eqref{eq:degreefourier} in \cite[Theorem 12.6]{BM-book}.
\end{proof}

This motivates the following definition.
\begin{definition}[Defect]\label{def:defect}
For a smooth map $v\colon \S^1\to\S^1$ given as its Fourier series
\[
v(e^{\i t})=\sum_{k\in\Z}a_k e^{\i kt},
\]
we define the defect
\[
D(v)
\coloneqq
\sum_{k<0}|k|\,|a_k|^2.
\]
\end{definition}

Then \Cref{la:E12vsdefect} provides the well-known decomposition
\begin{equation}\label{eq:hhvsdefect}
E_{1/2}(v)=4\pi^2\deg v+8\pi^2 D(v) \quad \text{if $\deg v \geq 0$},
\end{equation}
cf. \cite[Theorem 5.3.]{BFavorite}.

It is worth pointing out the regularity needed to make sense of $D(u)$
\begin{lemma}\label{la:continuityBuH12}
We have
\[
 W^{\frac{1}{2},2}(\S^1,\S^1) \ni u \mapsto D(u)
\]
is locally uniformly continuous in the sense that
\[
 \abs{D(u_1)-D(u_2)} \aleq \brac{[u_1]_{W^{\frac{1}{2},2}(\S^1)} + [u_2]_{W^{\frac{1}{2},2}(\S^1)}} [u_1-u_2]_{W^{\frac{1}{2},2}}.
\]
\end{lemma}
\begin{proof}
This follows from \eqref{eq:hhvsdefect}, observing that
\[
 \abs{\deg(v_1)-\deg(v_2)} \aleq [v_1-v_2]_{H^{\frac{1}{2}}}\, \brac{[v_1]_{H^{\frac{1}{2}}}+ [v_2]_{H^{\frac{1}{2}}}}.
\]
Which follows from the degree formula and integration by parts.
\end{proof}

The defect $D(u)$ controls the difference to Blaschke products in the following sense.
\begin{lemma}\label{la:zero-or-equality-defect-blaschke}
Let
$u\in W^{\frac12,2}(\S^1,\S^1)$. Assume that its defect from \Cref{def:defect} satisfies $D(u)=0$.
Then $u$ agrees a.e. on $\S^1$ with a finite Blaschke product:
there exist $\theta\in\mathbb R$, an integer $d\ge0$, and points
\[
\alpha_1,\ldots,\alpha_d\in\mathbb D
\]
such that
\[
u(z)
=
e^{\i\theta}
\prod_{j=1}^d
\frac{z-\alpha_j}{1-\overline{\alpha_j}z}
\qquad\text{for a.e. }z\in\S^1.
\]
If $d=0$, then $u$ is constant.
\end{lemma}
\begin{proof}
This is also well-known, but we discuss the proof for convenience.

If $D(u)=0$ then by the definition of $D$ we can write
\[
 u(z) = \sum_{k=0}^\infty a_k z^k, \qquad z \in \S^1.
\]
Denote by $U$ the holomorphic extension which has the same formula
\[
 U(z) = \sum_{k=0}^\infty a_k z^k, \qquad |z| \leq 1.
\]
We now compute the boundary normal derivative. Since
\[
U(re^{\i t})=\sum_{k\ge0}a_k r^k e^{\i kt},
\]
we have
\[
\partial_\nu U(e^{\i t})
=
\partial_r U(re^{\i t})\big|_{r=1}
=
\sum_{k\ge1}k a_k e^{\i kt}.
\]
On the other hand, it holds
\[
\partial_t \left[ u(e^{\i t}) \right] 
=
\sum_{k\ge1}\i k a_k e^{\i kt}.
\]
Therefore we have
\[
\partial_\nu U =-\i\partial_t u
\qquad\text{on }\S^1.
\]
Thus we obtain
\[
 \begin{cases}
  \lap U = 0 \quad &\text{in ${\mathbb D}$}\\
  u \perp \partial_\nu U = 0 \quad &\text{on $\S^1$}.
 \end{cases}
\]
Hence $u$ is a half-harmonic map into $\S^1$, and by \cite[Theorem 4.25]{MS15}, $u$ is a Blaschke product.
\end{proof}

From \cite{DSW23}, we recall the following stability result.
\begin{lemma}[The defect controls the distance to Blaschke products]
\label{la:B-controls-Blaschke}
Let $u\in \dot H^{\frac12}(\S^1,\S^1)$ with positive degree $\deg u = d\geq 1$. Let $D$ be the defect from \Cref{def:defect}.

\begin{itemize}
\item If $d = 1$, then there exists a universal
constant $C>0$ such that
\[
\inf_{B_0\in{\rm B}_1^+}
\|u-B_0\|_{\dot H^{\frac12}(\S^1)}^2
\le
C\, D(u).
\]
\item if $d\geq 2$, assuming additionally
\begin{equation}\label{hyp:smallness}
\inf_{B_0\in{\rm B}_d^+}\|u-B_0\|_{\dot H^{\frac12}(\S^1)}^2 < \eps_{d},
\end{equation}
we get similarly
\begin{equation}
 \inf_{B_0\in{\rm B}_d^+}
\|u-B_0\|_{\dot H^{\frac12}(\S^1)}^2
\le
C\, D(u).
\end{equation}
\end{itemize}
\end{lemma}

\begin{proof}
From \eqref{eq:1/2energyfourier} we deduce that
\begin{equation}\label{eq:energy_zk}
E_{1/2}(e^{\i kt}) = 4\pi^2|k|.
\end{equation}
Thus, by \eqref{eq:hhvsdefect} and \eqref{eq:energy_zk},
\[
 8 \pi^2 D(v) = E_{\frac{1}{2}}(v)-E_{\frac{1}{2}}(z^k)
\]
By \cite[Theorem 1.3]{DSW23} for any degree $1$ map $v$
\[
 8 \pi^2 D(v) = E_{\frac{1}{2}}(v)-E_{\frac{1}{2}}(z) \geq c \inf_{B_0\in{\rm B}_1^+}
\|u-B_0\|_{\dot H^{\frac12}(\S^1)}^2
\]
and for any degree $d \geq 2$, by \cite[Theorem 1.5]{DSW23}, under the assumption \eqref{hyp:smallness}
\[
 8 \pi^2 D(v) = E_{\frac{1}{2}}(v)-E_{\frac{1}{2}}(z^k) \geq c \inf_{B_0\in{\rm B}_k^+}
\|u-B_0\|_{\dot H^{\frac12}(\S^1)}^2
\]
For degree-one maps, the half-harmonic energy satisfies
\[
E_{\frac12}(u)
=
4\pi^2\deg u+8\pi^2 D(u).
\]
Since $\deg u=1$
and $E_{\frac12}(\id)=4\pi^2$, we obtain
\[
E_{\frac12}(u)-E_{\frac12}(\id)
=
8\pi^2 D(u).
\]
By the quantitative stability theorem for degree-one half-harmonic maps,
there exists a universal constant $C>0$ such that
\[
\inf_{B_0\in{\rm B}_1^+}
\|u-B_0\|_{\dot H^{\frac12}(\S^1)}^2
\le
C\left(
E_{\frac12}(u)-E_{\frac12}(\id)
\right).
\]
Using the identity above gives
\[
\inf_{B_0\in{\rm B}_1^+}
\|u-B_0\|_{\dot H^{\frac12}(\S^1)}^2
\le
C\, D(u).
\]
This proves the lemma.
\end{proof}

We will need to sum over $D(v^n)$ for smooth maps $v\colon \S^1 \to \S^1$. To make the sums convergent, we observe
\begin{lemma}\label{la:Bvroughestimate}
Let $v \in \lip(\S^1,\S^1)$ then for any $n \geq 1$
\[
 D(v^n) \aleq C_{v} n
\]
\end{lemma}
\begin{proof}
We have by definition
\[
 D(v^n) \leq E_{1/2} (v^n).
\]
By interpolation, we obtain
\[
 [v^n]_{W^{\frac{1}{2},2}} \aleq \|v^n\|_{L^2}^{\frac{1}{2}} \|\nabla (v^n)\|_{L^2}^{\frac{1}{2}} \aleq \sqrt{n} \|\nabla v\|_{L^{\infty}}^{\frac{1}{2}}.
\]
Thus it holds
\[
 E_{\frac{1}{2}}(v^n) \aleq n \|v\|_{\lip}^2.
\]
\end{proof}

\begin{lemma}[Defect inequality under powers]\label{la:defect}
Let $v \in W^{\frac{1}{2},2}(\S^1,\S^1)$. Taking the defect $D$ from \Cref{def:defect}, for every integer $n\ge1$,
\[
D(v^n)\le n^2\, D(v).
\]
If for $n\geq 2$ we have equality $D(v^n) = n^2\, D(v)$ then $v$ is a Blaschke product.
\end{lemma}

\begin{proof}
Observe that if $v \in W^{\frac{1}{2},2}(\S^1,\S^1)$ then $v^n \in W^{\frac{1}{2},2}(\S^1,\S^1)$, and by approximation we may first assume $v$ is smooth and take the limit to make every inequality hold for $v$. We write
\[
\begin{split}
v(e^{\i t})
&= \sum_{k\ge0}a_k e^{\i kt}
+
\sum_{m\ge1}a_{-m}e^{-\i mt}
=\sum_{k\ge0}a_k e^{\i kt} + \sum_{m\ge1}a_{-m}\overline{e^{\i mt}}.
\end{split}
\]
Set
\begin{equation}\label{eq:Vexpansion}
 V(z) \coloneqq \sum_{k\ge0}a_k z^{k}
+
\sum_{m\ge1}a_{-m}\overline{z^{m}}.
\end{equation}
Then $V\colon \mathbb{D} \to \C$ verifies $\partial_z \partial_{\bar z} V=0$ (with $\partial_{\bar{z}} \coloneqq   \frac{1}{2}  \brac{\partial_{x} + \i \partial_y}$ and $\partial_z =\frac{1}{2}  \brac{\partial_{x} - \i \partial_y}$). Hence it is the harmonic extension of $v\colon \partial \mathbb{D} \to \C$. In particular we have from the maximum principle $\sup_{\mathbb{D}} |V| \leq \sup_{\partial \mathbb{D}} |v| = 1$.
We have
\[
 \partial_{\bar{z}} V = \sum_{m \geq 1} a_{-m} m \bar{z}^{m-1}.
\]
Moreover for $\ell,k \geq 0$, it holds
\[
 \int_{\mathbb{D}} \bar{z}^\ell z^k = \int_0^1 r^{1+\ell + k} \int_{0}^{2\pi}
 \overline{e^{\i \theta \ell}} e^{\i \theta k} \dd\theta \dd r  = \frac{2\pi}{2+\ell+k} \delta_{\ell k}.
\]
Thus, recalling \Cref{def:defect}, we obtain
\begin{equation}\label{eq:BvvsHolom}
\int_{\mathbb{D}}|\partial_{\bar{z}} V|^2 = \sum_{m \geq 1} |a_{-m}|^2 |m|^2 \frac{2\pi}{2m}
=
\pi
\sum_{m\ge1}m|a_{-m}|^2
=
\pi D(v).
\end{equation}
Now fix $n \geq 1$, let $\tilde{V}_n\colon \mathbb{D} \to \C$ be the harmonic extension as above of $v^n\colon \S^1 \to \S^1$, then by \eqref{eq:BvvsHolom} we have
\[
 \pi D(v^n) = \int_{\mathbb{D}}|\partial_{\bar{z}} \tilde{V}_n|^2 \overset{\eqref{eq:holomextvsdeg}}{=} \frac{1}{4} \int_{\mathbb{D}} |\nabla \tilde{V}_n|^2 - \frac{\pi}{2} \deg (v^n).
\]
Since $\tilde{V}_n$ is the harmonic extension, the Dirichlet energy can be estimate from above, observe that $V^n\colon \mathbb{D} \to \C$ is another extension of $v^n$, and $|V^n| =|V|^n \leq 1$. Hence we have
\begin{equation}\label{eq:ineqB}
\begin{split}
\pi D(v^n) &\leq \frac{1}{4} \int_{\mathbb{D}} |\nabla V^n|^2 - \frac{\pi}{2} \deg v^n\\[2mm]
 &\overset{\eqref{eq:holomextvsdeg}}{=} \int_{\mathbb{D}}|\partial_{\bar{z}} V^n|^2\\[2mm]
 & =\int_{\mathbb{D}}|n V^{n-1} \partial_{\bar{z}} V|^2\\[2mm]
& \overset{n \geq 1}{\leq} n^2 \int_{\mathbb{D}}| \partial_{\bar{z}} V|^2\\[2mm]
& \overset{\eqref{eq:BvvsHolom}}{=} n^2 \pi D(v).
\end{split}
 \end{equation}
Assume now that for some $n\ge2$ equality holds, i.e.\ $D(v^n)=n^2D(v)$. Then by \eqref{eq:ineqB}, we have
\[
\int_{\mathbb D}
|V|^{2n-2}|\partial_{\bar z}V|^2
=
\int_{\mathbb D}
|\partial_{\bar z}V|^2 ,
\]
that is \[
\int_{\mathbb D}
\left(1-|V|^{2n-2}\right)
|\partial_{\bar z}V|^2
=
0.
\]
Since by strong maximum principle $|V| < 1$ inside of $\mathbb{D}$, we conclude $\partial_{\bar{z}} V =0$, i.e., $V$ is holomorphic. From the definition of $V$ \eqref{eq:Vexpansion} we get $a_{-m} = 0$ for all $m \geq 1$. Thus $v(e^{\i t})
= \sum_{k\ge0}a_k e^{\i kt}$ and with \Cref{def:defect} we obtain $D(v) = 0$. Thus $v$ is a finite Blaschke product by \Cref{la:zero-or-equality-defect-blaschke}.
\end{proof}

\subsection{Fourier expansion}

A substantial part of our proof rests on the following Fourier identity.
\begin{proposition}\label{pr:Fourierexpansion}
Let $p\in (1,\infty)$. There are real coefficients $\gamma_{p,n}$, depending only on $p$
and $n$, such that for any $\theta \in \R$
\begin{equation}\label{eq:eithetaexp}
|e^{i\theta}-1|^p
=
\sum_{n=1}^{\infty}
\gamma_{p,n}(1-\cos(n\theta)).
\end{equation}
They are given by
\[
\gamma_{p,n}
=
\frac{-2(-1)^n \Gamma(p+1)}
{
\Gamma\left(1+\frac p2+n\right)
\Gamma\left(1+\frac p2-n\right)
}=
\begin{cases} 
	\displaystyle \frac{2}{\pi}
\frac{\Gamma(p+1)}
{
\Gamma\left(1+\frac p2+n\right)
}
\sin\left(\pi \frac{p}{2} \right)
\Gamma\left(n-\frac p2\right)  &\text{ if }p \neq 2,\\[4mm]
2  &\text{ if }p=2 \text{ and }n=1,\\
0  &\text{ if }p=2 \text{ and }n\geq 2.
\end{cases}
\]
If the second Gamma factor has a pole, the coefficient is interpreted as
$0$.

We have the following convergence for any $p \in (1,\infty)$ and $\sigma<p$,
\begin{equation}\label{eq:nsquaregamman}
 \sum_{n =1}^\infty n^\sigma \abs{\gamma_{p,n}} < \infty.
\end{equation}
For $1<p<2$,
\begin{equation}\label{eq:gammasign1lpl2}
\gamma_{p,n}>0
\qquad\text{for every }n\ge1.
\end{equation}
For $p=2$,
\begin{equation}\label{eq:gammasignpeq2}
\gamma_{2,1}=2,
\qquad
\gamma_{2,n}=0
\qquad\text{for every }n\ge2.
\end{equation}
For $2<p<4$,
\begin{equation}\label{eq:gammasign2lpl4}
\gamma_{p,1}>0,
\qquad
\gamma_{p,n}<0
\qquad\text{for every }n\ge2.
\end{equation}

Finally, for every integer $\ell\ge1$ such that $2\ell<p$,
\begin{equation}\label{eq:momentidentities}
\sum_{n=1}^{\infty}n^{2\ell}\gamma_{p,n}=0.
\end{equation}
and in particular, for $2<p\le4$,
\begin{equation}\label{eq:gammasignsumeeq0}
\sum_{n=1}^{\infty}n^2\gamma_{p,n}=0.
\end{equation}
\end{proposition}
\begin{proof}
For $n \geq 1$, we define
\[
\begin{split}
 \beta_{p,n} \coloneqq & \frac{1}{\pi} {\int_{-\pi}^\pi |e^{\i \sigma}-1|^p \cos(n\sigma) \dd\sigma}.
\end{split}
 \]
We have
\[
|e^{i\theta}-1|^2
=
2-2\cos\theta
=
4\sin^2\frac{\theta}{2}.
\]
We obtain with \eqref{eq:trigintegral}
\[
\begin{split}
 \beta_{p,n} &=  \frac{2^{p}}{\pi} {\int_{-\pi}^\pi \left| \sin\brac{\frac{\sigma}{2}} \right|^p \cos(n\sigma) \dd\sigma}\\[2mm]
 &=  \frac{2^{p+1}}{\pi} \int_{-\frac{\pi}{2}}^{\frac{\pi}{2}} \left|\sin\brac{\sigma}\right|^p \cos(2n\sigma) \dd\sigma\\[2mm]
 &=  \frac{2^{p+1}}{\pi} \int_{0}^{\pi} \sin^p\brac{\sigma}\, \cos(2n\sigma) \dd\sigma\\[2mm]
&=  \frac{2^{p+1}}{\pi}\,
(-1)^n\,
\frac{
\pi\Gamma(p+1)
}{
2^p
\Gamma\left(1+\frac p2+n\right)
\Gamma\left(1+\frac p2-n\right)
}\\[2mm]
&=  
\frac{ 2
	(-1)^n
\Gamma(p+1)
}{
\Gamma\left(1+\frac p2+n\right)
\Gamma\left(1+\frac p2-n\right)
}.
 \end{split}
 \]
The function $\theta \mapsto |e^{\i\theta}-1|^p$ is even and $2\pi$-periodic. Thus, we can write it as a Fourier series
\[
 |e^{\i \theta}-1|^p = \beta_{0,p} +  \sum_{n=1}^\infty \beta_{p,n} \cos(n\theta).
\]
Evaluating this at $\theta = 0$ we thus have
\[
 \beta_{0,p} = -\sum_{n=1}^\infty \beta_{p,n}.
\]
This \emph{formally} leads to
\[
 |e^{\i \theta}-1|^p = \sum_{n=1}^\infty \beta_{p,n} (\cos(n\theta)-1)
\]
or, with
\[
\gamma_{p,n}
\coloneqq -\beta_{p,n} =
-2(-1)^n
\frac{\Gamma(p+1)}
{
\Gamma\left(1+\frac p2+n\right)
\Gamma\left(1+\frac p2-n\right)
}\\
\]
we have \eqref{eq:eithetaexp}.

For this argument to be precise, we need to ensure that
\[
 \sum_{n=1}^\infty \abs{\beta_{p,n}} < \infty.
\]
We use the Euler's reflection formula \eqref{eq:reflectionformula}, but only if $z=1+\frac p2-n \not \in \Z$, i.e.\ when $\frac{p}{2} \not \in \Z$. In that case we find, using that $\sin(a+(1-n) \pi) = (-1)^{1-n} \sin(a)$ for all $n \in \N$,
\[
\frac1{\Gamma\left(1+\frac p2-n\right)}
=
\frac{
\sin\left(\pi\left(1+\frac p2-n\right)\right)
}{
\pi
}
\Gamma\left(n-\frac p2\right)
=(-1)^{n+1}
\frac{
\sin\left(\pi \frac{p}{2}\right)
}{
\pi
}
\Gamma\left(n-\frac p2\right).
\]
Thus
\begin{equation}\label{eq:gammapnv2}
\beta_{p,n}= -
\frac{2}{\pi}
\frac{\Gamma(p+1)}
{
\Gamma\left(1+\frac p2+n\right)
}
\sin\left(\pi \frac{p}{2} \right)
\Gamma\left(n-\frac p2\right) \quad \text{ for } p \not \in 2\Z.
\end{equation}
Since $\abs{\sin\left(\pi\left(1+\frac p2-n\right)\right)}
=\abs{\sin\left(\frac{\pi p}{2}\right)}$
we obtain, whenever the formula has no pole (in which case $\beta_{p,n} = 0$)
\[
\abs{\beta_{p,n}}
\aleq
\abs{\frac{
\Gamma(p+1)\sin\left(\frac{\pi p}{2}\right)
\,
\Gamma\left(n-\frac p2\right)
}{
\Gamma\left(n+1+\frac p2\right)
}} \aleq \abs{\frac{
\Gamma\left(n-\frac p2\right)
}{
\Gamma\left(n+1+\frac p2\right)
}}.
\]
For large $n$, using \eqref{eq:prodformula}, we have
\begin{equation}\label{eq:ratio_Gamma}
\frac{
\Gamma\left(n-\frac p2\right)
}{
\Gamma\left(n+1+\frac p2\right)
}
\aleq
n^{-p-1}.
\end{equation}
Thus certainly for any $p>1$, it holds
\[
 \sum_{n =1}^\infty \abs{\beta_{p,n}} < \infty.
\]
In particular, the estimate \eqref{eq:ratio_Gamma} also implies \eqref{eq:nsquaregamman}.

\underline{Now we prove the estimates \eqref{eq:gammasign1lpl2}, \eqref{eq:gammasignpeq2}, \eqref{eq:gammasign2lpl4}}.
From \eqref{eq:gammapnv2}, we have
\[
\gamma_{p,n}=
\frac{2}{\pi}
\frac{\Gamma(p+1)}
{
\Gamma\left(1+\frac p2+n\right)
}
\sin\left(\pi \frac{p}{2} \right)
\Gamma\left(n-\frac p2\right).
\\
\]
If \underline{$1<p<2$}, then
for every $n\ge1$,
\[
n-\frac p2>0.
\]
We also have
\[
\sin\left(\frac{\pi p}{2}\right)>0.
\]
Therefore, it holds
\[
\gamma_{p,n}>0
\qquad\text{for every }n\ge1.
\]
This gives \eqref{eq:gammasign1lpl2}

If \underline{$p=2$}, then directly
\[
|e^{i\theta}-1|^2
=
2(1-\cos\theta).
\]
Thus
\[
\gamma_{2,1}=2,
\qquad
\gamma_{2,n}=0
\qquad\text{for }n\ge2.
\]
This is \eqref{eq:gammasignpeq2}.

If \underline{$2<p<4$}, we observe $\sin\left(\frac{\pi p}{2}\right)<0$. Since $\Gamma(1-\frac{p}{2}) < 0$, and $\Gamma (n-\frac{p}{2}) > 0$ for $n \geq 2$ we find \eqref{eq:gammasign2lpl4}.

Finally, since \eqref{eq:nsquaregamman} and $p>2$ we have $\sum_{n=1}^{\infty}n^2|\gamma_{p,n}|<\infty$, and thus we can differentiate both sides of \eqref{eq:eithetaexp} at $\theta =0$, and find
\[
 0 = \sum_{n=1}^\infty \gamma_{p,n} n^2.
\]
This \eqref{eq:gammasignsumeeq0}.
\end{proof}

The previous lemma leads to the following elementary representation. Elementary as it is, it is the crucial observation needed to pass from $L^p$ to $L^2$.
\begin{lemma}\label{la:PointwisePowerIdentity}
Let $p>1$, and let $\gamma_{p,n}$ be the coefficients from
\Cref{pr:Fourierexpansion}. If $a,b\in\S^1$, then
\[
|a-b|^p
=
\frac12
\sum_{n=1}^{\infty}
\gamma_{p,n}|a^n-b^n|^2.
\]
\end{lemma}

\begin{proof}
We write $a=e^{\i\alpha}$ and $b=e^{\i\beta}$. Then we have
\[
|a-b|^p
=
|e^{\i(\alpha-\beta)}-1|^p.
\]
By Lemma \ref{pr:Fourierexpansion}, we obtain 
\[
|e^{\i(\alpha-\beta)}-1|^p
=
\sum_{n=1}^{\infty}
\gamma_{p,n}
\left(
1-\cos(n(\alpha-\beta))
\right).
\]
Moreover it holds
\[
|a^n-b^n|^2
=
|e^{\i n\alpha}-e^{\i n\beta}|^2
=
2-2\cos(n(\alpha-\beta)).
\]
\end{proof}

\section{Sobolev regularity for local minimizers}\label{s:regularity}
Throughout this section, we use
\[
        E_s(v)
        \coloneqq
        \int_{\S^1}\int_{\S^1}
        \frac{|v(x)-v(y)|^p}{|x-y|^2}
        \dd\ell(x)\dd\ell(y),
        \qquad
        p=\frac1s.
\]

The following is an adaptation of the arguments in \cite{BL17}, and the H\"older regularity theory from, e.g., \cite{Sucks23}. Observe that the threshold $s>\frac{1}{4}$ is \emph{better} than the estimate one might expect from \cite{BL17}. This is due to the fact that we work globally on $\S^1$, and thus the ``tail space regularity'', which is kept constant in \cite{BL17} is improving in every step of the bootstrap, with this in mind we indeed obtain the same level of regularity as  expected from \cite{BL17}. What is curious is that the threshold $s>\frac{1}{4}$ in \Cref{th:minH12} seems to be completely independent of the threshold $s>\frac{1}{4}$ in \Cref{th:blaschkemin}.

\begin{theorem}[Minimizers belong to $H^{1/2}$]
\label{th:minH12}
Let $\frac14<s<1$, $p=\frac1s$.
Let $u\in W^{s,p}(\S^1,\S^1)$ be a minimizer of $E_s$ among maps $v\colon \S^1\to\S^1$ with the same degree as $u$.
Then $u\in H^{1/2}(\S^1,\mathbb R^2)$.
\end{theorem}
For $s \geq \frac{1}{2}$ \Cref{th:minH12} follows from Sobolev embedding. For $s \in (0,\frac{1}{2})$ we first recall the following result of \cite{Sucks23}, which means that the a priori critical Euler--Lagrange equation becomes subcritical and only bootstrapping is needed.
\begin{theorem}\label{th:heolderreg}
Let $u$ be as in \Cref{th:minH12}. Then for some $\eps > 0$ we have $u \in C^\eps$.
\end{theorem}
Observe that there is no estimate of the $C^\eps$-norm of $u$, since the energy is conformally invariant the $C^\eps$-norm can be arbitrarily large.
\begin{definition}
\label{def:nikolskii-seminorm}
For $0<\gamma<1$, $1\le p<\infty$ and for a $2\pi$-periodic function $f$, we define
\[
        [f]_{\mathcal N^{\gamma,p}}
        \coloneqq
        \sup_{0<|h|<1}
        \frac{\|f(\cdot+h)-f(\cdot)\|_{L^p(0,2\pi)}}{|h|^\gamma}.
\]
\end{definition}
The following is proven in \cite[Proposition 2.6]{BL17}.
\begin{lemma}[Sobolev implies Nikol'skii]
\label{lem:Sobolev-implies-Nikolskii}
Let $0<\gamma<1$ and $1\le p<\infty$.
If $f\in W^{\gamma,p}(\S^1)$, then
\[
        [f]_{\mathcal N^{\gamma,p}}
        \le
        C [f]_{W^{\gamma,p}(\S^1)}.
\]
\end{lemma}

The following is just a (rough) Poincar\'e type inequality and can be easily proven by hand using the Gagliardo seminorm.
\begin{lemma}[Nikol'skii regularity above $1/2$ implies $H^{1/2}$]
\label{lem:Nikolskii-above-half-implies-Hhalf}
Let $p>2$ and $\gamma>\frac12$. If $f\in L^p(\S^1)$ with $[f]_{\mathcal N^{\gamma,p}}<\infty$,
then $f\in H^{1/2}(\S^1)$.
\end{lemma}

The following is proven in \cite{MMS26}
\begin{lemma}[Lifting]
\label{lem:uniform-small-arc-lifting}
Let $\alpha>0$ and $\lambda>0$.
There exist constants $r_{{lift}}>0$ and $C_{{lift}}>0$, depending only on $\alpha$ and $\lambda$ such that the following holds.
Let $u\in C^\alpha(\S^1,\S^1)$ satisfying
\[
        [u]_{C^\alpha(\S^1)}\le \lambda.
\]
If $I\subset\S^1$ is an arc of length $0<|I|\le r_{{lift}}$, then there exists a unique lift $\varphi\colon I \to [-\frac{\pi}{4},\frac{\pi}{4}]$ such that
\[
	\begin{cases}
        u=e^{\i\varphi} \qquad \text{on }I, \\[2mm]
        [\varphi]_{C^\alpha(I)} \le C_{{lift}}(1+\lambda).
    \end{cases}
\]
\end{lemma}

The following is the counterpart to the underlying argument in the regularity theory for $p\leq 2$ used in \cite[Lemma 6.4]{MMS26}
\begin{lemma}[convexity and semiconvexity for $p>2$]
\label{lem:superquadratic-chordal-convexity}
Let $p>2$. There exist constants $\eta_{{cvx}}>0$, $c_{{cvx}}>0$ and $C_{{sc}}>0$ depending only on $p$, such that the following hold. If $\xi,\zeta\in [-\eta_{\rm cvx}, \eta_{\rm cvx}]$, then it holds
\begin{equation}\label{eq:convexity1st}
\begin{aligned}
        &|e^{\i\xi}-1|^p
        +
        |e^{\i\zeta}-1|^p
        -
        2\left|e^{\i\frac{\xi+\zeta}{2}}-1\right|^p \ge
        c_{{cvx}}|\xi-\zeta|^p.
\end{aligned}
\end{equation}
Moreover, for every $\omega\in\S^1$ and every $\xi,\zeta\in\mathbb R$, we have
\begin{equation}\label{eq:convexityglobal}
\begin{aligned}
        &|e^{\i\xi}-\omega|^p
        +
        |e^{\i\zeta}-\omega|^p
        -
        2\left|e^{\i\frac{\xi+\zeta}{2}}-\omega\right|^p
        \ge
        -C_{{sc}}|\xi-\zeta|^2.
\end{aligned}
\end{equation}
\end{lemma}

\begin{proof}
We define
\[
        f(t)\coloneqq |e^{\i t}-1|^p
        =
        \left(2\left|\sin\frac t2\right|\right)^p.
\]
For $0<|t|<\pi$, direct differentiation gives
\[
        f''(t)
        =
        p\left(2\left|\sin\frac t2\right|\right)^{p-2}
        \left[
        (p-1)\cos^2\frac t2
        -
        \sin^2\frac t2
        \right].
\]
Choose $\eta_{{cvx}}\in(0,\frac12)$ so small that
\[
        \sin^2\frac{\eta_{{cvx}}}{2}
        \le
        \frac{p-1}{2p}.
\]
Then for $0<|t|\le\eta_{{cvx}}$, it holds
\[
\begin{aligned}
        (p-1)\cos^2\frac t2-\sin^2\frac t2
        &=
        (p-1)-p\sin^2\frac t2 \ge
        \frac{p-1}{2}.
\end{aligned}
\]
Also, for $0<|t|\le\eta_{{cvx}}$, we have
\[
        2\left|\sin\frac t2\right|\ge c|t|.
\]
Since $p-2>0$, we get
\[
        f''(t)\ge c|t|^{p-2}
        \qquad
        \text{for }0<|t|\le \eta_{{cvx}}.
\]
Now let us prove the estimates. If $\xi = \zeta$ there is nothing to prove, so from now on assume $\xi \neq \zeta$, i.e. $\abs{\frac{\xi-\zeta}{2}}>0$.

We begin with \underline{\eqref{eq:convexity1st}:}
Since $p>2$, the function $f$ is $C^2$ near $0$, with $f''(0)=0$. Therefore
\[
\begin{aligned}
        &f(\xi)+f(\zeta)-2f\left(\frac{\xi+\zeta}{2}\right)
        =
        \int_{-1}^{1}
        (1-|\sigma|)
        f''\left(
        \frac{\xi+\zeta}{2}
        +
        \sigma\frac{\xi-\zeta}{2}
        \right)
        \left(\frac{\xi-\zeta}{2}\right)^2
        \dd\sigma.
\end{aligned}
\]
The arguments of $f''$ lie in $[-\eta_{{cvx}},\eta_{{cvx}}]$. Thus
\[
\begin{aligned}
        &f(\xi)+f(\zeta)-2f\left(\frac{\xi+\zeta}{2}\right) \ge
        c\abs{\frac{\xi-\zeta}{2}}^2
        \int_{-1}^{1}
        (1-|\sigma|)
        \left|
        \frac{\xi+\zeta}{2}
        +
        \sigma\frac{\xi-\zeta}{2}
        \right|^{p-2}
        \dd\sigma.
\end{aligned}
\]
We claim that for all $A\in\mathbb R$, if $\abs{\frac{\xi-\zeta}{2}}>0$ then
\[
        \int_{-1}^{1}
        (1-|\sigma|)
        \left|
        A+\sigma \abs{\frac{\xi-\zeta}{2}}
        \right|^{p-2}
        \dd\sigma
        \ge
        c\abs{\frac{\xi-\zeta}{2}}^{p-2}.
\]
 Dividing by $\abs{\frac{\xi-\zeta}{2}}$, it is enough to prove
\[
        \int_{-1}^{1}(1-|\sigma|)|a+\sigma|^{p-2}\dd\sigma\ge c.
\]
If $|a|\ge2$, then $|a+\sigma|\ge1$ for all $\sigma\in[-1,1]$, and the integral is at least
\[
        \int_{-1}^{1}(1-|\sigma|)\dd\sigma=1.
\]
If $|a|\le2$, the following function is continuous and strictly positive on the compact interval $[-2,2]$
\[
        a\mapsto
        \int_{-1}^{1}(1-|\sigma|)|a+\sigma|^{p-2}\dd\sigma.
\]
Hence it has a positive minimum. This proves the claim. Therefore
\[
        f(\xi)+f(\zeta)-2f\left(\frac{\xi+\zeta}{2}\right)
        \ge
        c\abs{\frac{\xi-\zeta}{2}}^p
        =
        c|\xi-\zeta|^p.
\]
This proves the local convexity estimate \eqref{eq:convexity1st}

For \underline{\eqref{eq:convexityglobal}}, we use the same formula for $f''$. Since
\[
        (p-1)\cos^2\frac t2-\sin^2\frac t2\ge -1
\]
and
\[
        \left(2\left|\sin\frac t2\right|\right)^{p-2}\le 2^{p-2},
\]
we obtain for $0<|t|<\pi$,
\[
        f''(t)\ge -p2^{p-2}.
\]
Since $f$ is $C^2$ at the points of $2\pi\mathbb Z$ for $p>2$, this lower bound holds in the distributional sense on every interval. Hence the function $t\mapsto f(t)+C_{{sc}}t^2$ is convex on every interval of length $2\pi$, for $C_{{sc}}$ large enough depending only on $p$. Therefore, it holds
\[
        f(\xi)+f(\zeta)-2f\left(\frac{\xi+\zeta}{2}\right)
        \ge
        -C_{{sc}}|\xi-\zeta|^2.
\]
Rotating the target gives the same estimate with $1$ replaced by any $\omega\in\S^1$. This prove \eqref{eq:convexityglobal}.
\end{proof}

The following was proven in \cite[Lemma 6.3]{MMS26}
\begin{lemma}[Inner variations]
\label{lem:uniform-inner-variations}
There exist constants $r_g>0$, $c_g>0$ and $C_g>0$ such that the following holds.

Let $I\subset\S^1$ be an arc of length $0<r\le r_g$
Let $X\in C^\infty(\S^1,\mathbb R)$ such that $\operatorname{spt}X\Subset I$ and, in arclength coordinates,
\begin{equation}\label{eq:Xestimates}
	|X|\le C_g,
	\qquad
	\left|\frac{dX}{d\ell}\right|\le \frac{C_g}{r},
	\qquad
	\left|\frac{d^2X}{d\ell^2}\right|\le \frac{C_g}{r^2} \qquad
	\left|\frac{d^3X}{d\ell^3}\right|\le \frac{C_g}{r^3}.
\end{equation}
For $|h|\le c_g r$, the map $\psi_h\colon \S^1\to \S^1$ defined by $\psi_h(e^{\i t})\coloneqq e^{\i(t+hX(e^{it}))}$ is an orientation-preserving diffeomorphism. Consequently, for every continuous map
$u\colon \S^1\to\S^1$ one has $\deg(u\circ\psi_h)=\deg u$.

Furthermore, let $p \in (1,\infty)$ and $s = \frac{1}{p} \in (0,1)$.
If $u\in W^{s,p}(\S^1,\S^1)$ is a minimizer of $E_{s}$ among maps $v\colon \S^1 \to \S^1$ with the same degree as $u$, then
\begin{equation}\label{eq:energyestimateucircpsih}
	0\le
	E_s(u\circ\psi_h)-E_s(u)
	\le
	C_g\, r^{-2}E_s(u)h^2.
\end{equation}
\end{lemma}

The following is the basis for the improved regularity
\begin{proposition}
\label{prop:localized-superquadratic-difference-quotient}
Let $2<p<\infty$, identify $s=\frac1p \in (0,\frac{1}{2})$ and fix
$\alpha,\lambda>0$. There exist constants $R_{{outer}}>0$, $\varepsilon_{{sep}}>0$, $C>0$,
depending only on $p,\alpha,\lambda$, such that the following holds.
Let
$u\in C^\alpha(\S^1,\S^1)\cap W^{s,p}(\S^1,\S^1)
$
satisfying
\[
        [u]_{C^\alpha(\S^1)}
        +
        [u]_{W^{s,p}(\S^1)}
        +
        E_s(u)
        \le
        \lambda.
\]
Assume that $u$ minimizes $E_s$ among maps $\S^1\to\S^1$ of the same degree.
Fix $x_0\in\S^1$. Let
\[
        0<r<\rho<R<R_{{outer}},
        \qquad
        r\le \frac{\rho}{4},
        \qquad
        \frac{\rho}{R-\rho}\le \varepsilon_{{sep}}.
\]
Choose a lift from \Cref{lem:uniform-small-arc-lifting}, $u=e^{\i\varphi}$ on $B(x_0,R)$.
Let $X\in C^\infty(\S^1,\mathbb R)$ satisfying $X=1$ on $B(x_0,r)$, $\operatorname{spt}X\Subset B(x_0,\rho)$ and, in arclength coordinates,
\[
        |X|\le C,
        \qquad
        \left|\frac{dX}{d\ell}\right|\le \frac C\rho,
        \qquad
        \left|\frac{d^2X}{d\ell^2}\right|\le \frac C{\rho^2},
        \qquad
        \left|\frac{d^3X}{d\ell^3}\right|\le \frac C{\rho^3}.
\]
For $|h|\le c\rho$ we define
\[
        \psi_h(e^{\i t})\coloneqq e^{\i(t+hX(e^{\i t}))}.
\]
We define
\[
        \theta_h(x)\coloneqq \varphi(\psi_h(x))-\varphi(x)
        \qquad\text{for }x\in B(x_0,R).
\]
Then it holds
\begin{equation}
\label{eq:localized-superquadratic-theta-estimate}
\begin{aligned}
        &\int_{B(x_0,R)}\int_{B(x_0,R)}
        \frac{|\theta_h(x)-\theta_h(y)|^p}{|x-y|^2}
        \dd\ell(x)\dd\ell(y) \le
        C\rho^{-2}h^2
        +
        \frac{C}{R-\rho}
        \int_{B(x_0,\rho)}|\theta_h(x)|^2\dd\ell(x).
\end{aligned}
\end{equation}
\end{proposition}

\begin{proof}
Choose $R_{{outer}}$ smaller than the lifting radius from
\Cref{lem:uniform-small-arc-lifting}, smaller than the geometric radius in
\Cref{lem:uniform-inner-variations}, and small enough that
\[
        C_{{lift}}(1+\lambda)R_{{outer}}^\alpha
        \le
        \eta_{{cvx}},
\]
where $\eta_{{cvx}}$ is from \Cref{lem:superquadratic-chordal-convexity}. By \Cref{lem:uniform-small-arc-lifting}, there is a lift $u=e^{\i\varphi}$ on $B(x_0,R)$ and
\[
        [\varphi]_{C^\alpha(B(x_0,R))}
        \le
        C_{{lift}}(1+\lambda).
\]
Since
\[
        |h|\le c\rho,
        \qquad
        \left|\frac{dX}{d\ell}\right|\le \frac C\rho,
\]
choosing $c$ small gives
\[
        |\psi_h(x)-\psi_h(y)|\le C|x-y|
        \qquad
        \forall x,y\in B(x_0,R).
\]
Thus, it holds
\begin{equation}
\label{eq:local-phase-smallness-superquadratic}
\begin{aligned}
        &|\varphi(x)-\varphi(y)|
        +
        |\varphi(\psi_h(x))-\varphi(\psi_h(y))| \le
        C_{{lift}}(1+\lambda)|x-y|^\alpha
        \le
        \eta_{{cvx}}.
\end{aligned}
\end{equation}
for all $x,y\in B(x_0,R)$.

Since $\operatorname{spt}X\Subset B(x_0,\rho)$, one has
\[
        \psi_h(x)=x
        \qquad
        \forall x\in B(x_0,R)\setminus B(x_0,\rho).
\]
Hence it holds
\[
        \theta_h(x)=0
        \qquad
        \forall x\in B(x_0,R)\setminus B(x_0,\rho).
\]
We define $v_h\coloneqq u\circ\psi_h$. On $B(x_0,R)$, we obtain $v_h(x)=e^{\i\varphi(\psi_h(x))}$.
We define for $x\in \S^1$,
\[
        m_h(x)\coloneqq
        \begin{cases}
        	\displaystyle e^{\i\frac{\varphi(x)+\varphi(\psi_h(x))}{2}} & \text{ if }x\in B(x_0,R),\\[2mm]
			u(x) & \text{ if }x\in\S^1\setminus B(x_0,R).
		\end{cases}
\]
Since $\theta_h=0$ near $\partial B(x_0,R)$, $m_h$ is globally well-defined. The map $v_h$ has the same degree as $u$ because $h \mapsto v_h$ is a homotopy, similarly the map $m_h$. By minimality, we have
$$
E_s(u)\le \min\Big( E_s(v_h), E_s(m_h) \Big)
$$
By \Cref{lem:uniform-inner-variations}, using $E_s(u)\le\lambda$, we obtain
\[
        E_s(v_h)-E_s(u)
        \aleq_\lambda
        \rho^{-2}h^2.
\]
Therefore, we end up with the estimate
\begin{equation}
\label{eq:local-superquadratic-deficit-upper}
        E_s(u)+E_s(v_h)-2E_s(m_h)
        \le
        C\rho^{-2}h^2.
\end{equation}
For $x,y\in B(x_0,R)$, using the lift, it holds
\begin{align*}
	& |u(x)-u(y)|^p = \left|e^{\i(\varphi(x)-\varphi(y))}-1\right|^p,\\[2mm]
	& |v_h(x)-v_h(y)|^p = \left|e^{\i(\varphi(\psi_h(x))-\varphi(\psi_h(y)))}-1\right|^p,\\[2mm]
	& |m_h(x)-m_h(y)|^p = \left| e^{\i\frac{\varphi(x)-\varphi(y)+\varphi(\psi_h(x))-\varphi(\psi_h(y))}2}
	-1 \right|^p.
\end{align*}
By \eqref{eq:local-phase-smallness-superquadratic} and
\Cref{lem:superquadratic-chordal-convexity}, specifically
\eqref{eq:convexity1st},
\begin{equation}
\label{eq:local-superquadratic-interior-lower}
\begin{aligned}
        &|u(x)-u(y)|^p
        +
        |v_h(x)-v_h(y)|^p
        -
        2|m_h(x)-m_h(y)|^p
        \\[2mm]
        &\qquad\ge
        c
        \left|
        \varphi(\psi_h(x))-\varphi(\psi_h(y))
        -
        \varphi(x)+\varphi(y)
        \right|^p
        \\[2mm]
        &\qquad=
        c|\theta_h(x)-\theta_h(y)|^p .
\end{aligned}
\end{equation}

Now take $ x\in B(x_0,R)$ and $y\in \S^1\setminus B(x_0,R)$.
If $x\in B(x_0,R)\setminus B(x_0,\rho)$, then $\psi_h(x)=x$, hence $v_h(x)=m_h(x)=u(x)$, and the contribution is zero. Thus only $x\in B(x_0,\rho)$ matters. For such $x$ and $y$,
$v_h(y)=m_h(y)=u(y)$, in particular $\theta_h(y) = 0$. Thus, by \Cref{lem:superquadratic-chordal-convexity}, specifically
\eqref{eq:convexityglobal},
\begin{equation}
\label{eq:local-superquadratic-cross-lower}
\begin{aligned}
        &|u(x)-u(y)|^p
        +
        |v_h(x)-u(y)|^p
        -
        2|m_h(x)-u(y)|^p \ge
        -C|\theta_h(x)|^2 .
\end{aligned}
\end{equation}

Combining \eqref{eq:local-superquadratic-deficit-upper},
\eqref{eq:local-superquadratic-interior-lower}, and
\eqref{eq:local-superquadratic-cross-lower}, and using symmetry of the cross terms, we obtain
\[
\begin{split}
        &\int_{B(x_0,R)}\int_{B(x_0,R)}
        \frac{|\theta_h(x)-\theta_h(y)|^p}{|x-y|^2}
        \dd\ell(x)\dd\ell(y)
        \\[2mm]
        \aleq &\int_{B(x_0,R)}\int_{B(x_0,R)}
        \frac{|u(x)-u(y)|^p}{|x-y|^2}
        \dd\ell(x)\dd\ell(y)
        +
        \int_{B(x_0,R)}\int_{B(x_0,R)}
        \frac{|v_h(x)-v_h(y)|^p}{|x-y|^2}
        \dd\ell(x)\dd\ell(y)\\[2mm]
        & - 2\int_{B(x_0,R)}\int_{B(x_0,R)} \frac{|m_h(x)-m_h(y)|^p}{|x-y|^2} \dd\ell(x)\dd\ell(y)\\[2mm]
        =& E_s(u)+E_s(v_h)-2E_s(m_h)\\[2mm]
        &-2\int_{\S \setminus B(x_0,R)}\int_{B(x_0,\rho)} \frac{|u(x)-u(y)|^p}{|x-y|^2} \dd\ell(x)\dd\ell(y)
        + 2\int_{\S \setminus B(x_0,R)}\int_{B(x_0,\rho)} \frac{|v_h(x)-v_h(y)|^p}{|x-y|^2} \dd\ell(x)\dd\ell(y)\\[2mm]
        & - 4\int_{\S \setminus B(x_0,R)}\int_{B(x_0,\rho)} \frac{|m_h(x)-m_h(y)|^p}{|x-y|^2} \dd\ell(x)\dd\ell(y)\\[2mm]
        \leq&C\rho^{-2}h^2 + C\int_{B(x_0,\rho)} \int_{\S^1\setminus B(x_0,R)} \frac{|\theta_h(x)|^2}{|x-y|^2} \dd\ell(y)\dd\ell(x).
\end{split}
\]
For $x\in B(x_0,\rho)$, it holds
\[
        \int_{\S^1\setminus B(x_0,R)}
        \frac{\dd\ell(y)}{|x-y|^2}
        \le
        \frac{C}{R-\rho}.
\]
Therefore, we have
\[
\begin{aligned}
        &\int_{B(x_0,R)}\int_{B(x_0,R)}
        \frac{|\theta_h(x)-\theta_h(y)|^p}{|x-y|^2}
        \dd\ell(x)\dd\ell(y) \le
        C\rho^{-2}h^2
        +
        \frac{C}{R-\rho}
        \int_{B(x_0,\rho)}|\theta_h(x)|^2\dd\ell(x).
\end{aligned}
\]
This proves \eqref{eq:localized-superquadratic-theta-estimate}.
\end{proof}

The following is the main gain in regularity bootstrap:
\begin{proposition}
\label{lem:local-bootstrap-step}
Let $s \in (0,\frac12)$, $p=\frac1s$. Let $u$ be as in \Cref{th:minH12}. Assume that for some $\gamma \in (0,1)$
\[
        [u]_{\mathcal N^{\gamma,p}}<\infty.
\]
Then for every $\mu \in (0,1)$ such that
\[
        0<\mu<\min\left\{s+\frac{2\gamma}{p},1\right\},
\]
one has
\[
        [u]_{\mathcal N^{\mu,p}}<\infty.
\]
\end{proposition}

\begin{proof}
Fix $x_0\in\S^1$. Choose $R$ suitably small so that \Cref{prop:localized-superquadratic-difference-quotient} is applicable, then we have from \eqref{eq:localized-superquadratic-theta-estimate} for $\theta_h$ as in \Cref{prop:localized-superquadratic-difference-quotient},
\[
\begin{aligned}
        &\int_{B(x_0,R)}\int_{B(x_0,R)}
        \frac{|\theta_h(x)-\theta_h(y)|^p}{|x-y|^2}
        \dd\ell(x)\dd\ell(y) \le
        C\rho^{-2}h^2
        +
        \frac{C}{R-\rho}
        \int_{B(x_0,\rho)}|\theta_h(x)|^2\dd\ell(x).
\end{aligned}
\]
Since $[u]_{\mathcal N^{\gamma,p}} < \infty$
this becomes
\[
\begin{aligned}
        &\int_{B(x_0,R)}\int_{B(x_0,R)}
        \frac{|\theta_h(x)-\theta_h(y)|^p}{|x-y|^2}
        \dd\ell(x)\dd\ell(y)
        \aleq_{R,u} \brac{|h|^2 + |h|^{2\gamma}}.
\end{aligned}
\]
Since this holds for all $B(x_0,R)$ and $R$ is uniform, we conclude that for all $|h| \ll 1$
\[
 [\delta_h u]_{W^{s,p}}^p \aleq \brac{|h|^2 + |h|^{2\gamma}}.
\]
So for $\mu \in (0,1)$, $\mu < s+\frac{2\gamma}{p}$ this implies $u \in \mathcal{N}^{\mu,p}$.
\end{proof}
\begin{proof}[Proof of \Cref{th:minH12}]
If $s\ge\frac12$ then the conclusion follows from the critical Sobolev embedding.
Assume now $\frac14<s<\frac12$, i.e.\ $p=\frac1s>2$. By \Cref{th:heolderreg}, $u\in C^\varepsilon(\S^1)$ for some $\varepsilon>0$. Let $\tilde{\gamma}_0\coloneqq s$ and
\[
        \tilde{\gamma}_{j+1}\coloneqq s+\frac{2\gamma_j}{p} = s +2s\gamma_j
\]
Then by induction
\[
        \tilde{\gamma}_j
        = \frac{s}{1-2s}
        -
        \left(
        \frac{s}{1-2s}-s
        \right)(2s)^j.
\]
In particular $\tilde{\gamma}_j$ is an increasing sequence since $2s < 1$, converging to $\frac{s}{1-2s}$. We have $\frac{s}{1-2s} > \frac{1}{2}$ if and only if $s > \frac{1}{4}$.
Thus, in the case $s \in (\frac{1}{4},\frac{1}{2})$ there exists an increasing sequence $\gamma_j$ with $\gamma_0 = s$ and
\[
 \gamma_{j+1}<s+\frac{2\gamma_j}{p}
\]
and for some finite $J \in \N$ we have $\gamma_J> \frac{1}{2}$.
Applying \Cref{lem:local-bootstrap-step} iteratively, we conclude that $[u]_{\mathcal{N}^{\gamma_J,p}} < \infty$.
Hence we conclude with \Cref{lem:Nikolskii-above-half-implies-Hhalf} that $u\in H^{1/2}(\S^1,\mathbb R^2)$.
\end{proof}

\section{Minimizers of any degree: Proof of Theorem~\ref{th:blaschkemin} and \ref{s:stability}}\label{s:proofs}
The backbone of our argument is the following representation formula.
\begin{proposition}\label{pr:Esvsid}
Let $u \in C^\infty(\S^1,\S^1)$ with $\deg u = d \geq 1$.
For any $s \in (0,1)$, we have 
\[
E_s(u)-\, E_s(z^d)
=
4\pi^2
\sum_{n=1}^{\infty}\gamma_{p,n}D(u^n),
\]
and the right-hand side is absolutely convergent. Here $D$ denotes the defect from \Cref{def:defect}, and $\gamma_{p,n}$ is defined in \Cref{pr:Fourierexpansion}.
\end{proposition}
\begin{proof}
The series on the right-hand side of the claimed is absolutely convergent: by \Cref{la:Bvroughestimate} we have $|D(u^n)| \aleq_{u} n$, and this implies convergence by \eqref{eq:nsquaregamman}.
By \Cref{la:PointwisePowerIdentity} we have
\[
|u(x)-u(y)|^p
=
\frac12
\sum_{n=1}^{\infty}
\gamma_{p,n}|u(x)^n-u(y)^n|^2.
\]
In particular for any $s = \frac{1}{p} \in (0,1)$, it holds
\begin{equation}\label{eq:EsinE12}
E_s(u)
=
\frac12
\sum_{n=1}^{\infty}
\gamma_{p,n}E_{1/2}(u^n).
\end{equation}
Combining \eqref{eq:hhvsdefect} and \Cref{la:degun}, we have $\deg(u^n)=nd$ and we find
\[
E_s(u)
=2\pi^2\, d\,
\sum_{n=1}^{\infty}
\gamma_{p,n}\, n
+
4 \pi^2
\sum_{n=1}^{\infty}
\gamma_{p,n}\, D(u^n).
\]
Observe that the first series converges for all $p \in (1,\infty)$ by \eqref{eq:nsquaregamman}.
On the other hand, for the identity, we have using \eqref{eq:EsinE12} and \Cref{la:energofzk}
\[
E_s(\id)
=
\frac12
\sum_{n=1}^{\infty}
\gamma_{p,n}E_{1/2}(\id^n) =
2\pi^2
\sum_{n=1}^{\infty}
\gamma_{p,n}  |n|
\]
and again convergence of this series follows for all $p \in (1,\infty)$ by \eqref{eq:nsquaregamman}.
Thus, we obtain 
\[
E_s(u) - d\, E_s(\id)
=4 \pi^2
\sum_{n=1}^{\infty}
\gamma_{p,n}\, D(u^n).
\]
We conclude with \Cref{la:energofzk}, $d\, E_s(\id)=E_s(z^d)$.
\end{proof}

The following in particular proves the first part of Theorem~\ref{th:blaschkemin}: $z^d$ (and all conformal transforms of it) are minimizers among all $W^{s,\frac{1}{s}}(\S^1,\S^1)$ maps of degree $d$. By density, it suffices to assume $u \in C^\infty(\S^1,\S^1)$.

\begin{corollary}\label{co:est}
Let $u \in C^\infty(\S^1,\S^1)$ with $\deg u = d$. If $p=\frac{1}{s} \in (1,2]$ then for some $c = c_{p}>0$ we have
\begin{equation}\label{eq:stability1}
E_s(u)-\, E_s(z^d) \geq c\, D(u) \geq 0.
\end{equation}
If $p \in (2,4)$ we have for some $c_p > 0$
\[
E_s(u)-\, E_s(z^d) \geq c_p (4D(u)-D(u^2)) \geq 0.
\]
Here $D$ denotes the defect from \Cref{def:defect}.
In particular, it holds
\begin{equation}\label{eq:zdismin}
 E_s(u) \geq E_s(z^d) \qquad \forall u \in C^\infty(\S^1,\S^1),\ \deg u = d
\end{equation}
\end{corollary}
\begin{proof}
 From \Cref{pr:Esvsid} we have
 \[
  E_s(u) - E_s(z^d) =4\pi^2
\sum_{n=1}^{\infty}\gamma_{p,n}D(u^n).
 \]
If $p \in (1,2]$, \Cref{pr:Fourierexpansion} \eqref{eq:gammasign1lpl2} we know $\gamma_{p,n}\ge0$, and $\gamma_{p,1} > 0$, so \eqref{eq:stability1} holds.
If $p \in (2,4)$, recall that by \eqref{eq:gammasign2lpl4}
\[
\gamma_{p,1}>0,
\qquad
\gamma_{p,n}\le0
\qquad\text{for }n\ge2.
\]
By \eqref{eq:gammasignsumeeq0}, we obtain
\[
\begin{split}
\sum_{n=1}^{\infty}\gamma_{p,n}\, D(u^n) 
= \sum_{n=1}^{\infty}\gamma_{p,n} \brac{D(u^n)-n^2 D(u)}
&= \sum_{n=2}^{\infty}\gamma_{p,n} \brac{D(u^n)-n^2 D(u)}.
\end{split}
\]
Since $\gamma_{p,n} \leq 0$ for $n \geq 2$ and $D(u^n)-n^2 D(u) \leq 0$ by \Cref{la:defect}, we obtain
\[
\begin{split}
&\sum_{n=1}^{\infty}\gamma_{p,n}D(u^n) 
\geq \gamma_{p,2}\, \brac{4 D(u)-D(u^2)}.
\end{split}
\]
Since $\gamma_{p,2}>0$ we conclude. The inequality \eqref{eq:zdismin} now follows since $D(u) \geq 0$ and by $4 D(u)-D(u^2)\geq 0$ (the latter by \Cref{la:defect}).
\end{proof}

The next corollary shows the second part of Theorem~\ref{th:blaschkemin}: any minimizer of $E_s$ among degree $d$ maps must be a Blaschke product
\begin{corollary}
Assume $s\in (\frac{1}{4},1)$.
Assume $u \in W^{s,\frac{1}{s}}(\S^1,\S^1)$ minimizes $E_s$ among degree $d$ maps. Then $u\in \mathscr{B}_d^+$.
\end{corollary}
\begin{proof}
Let first \underline{$p \in (1,2]$}. Then $u \in H^{\frac{1}{2}}(\S^1,\S^1)$ by Sobolev embedding, thus $D(u)$ is well-defined and continuous, by \Cref{la:continuityBuH12}. Thus, by smooth approximation from \Cref{co:est}, \eqref{eq:stability1},
\[
 D(u) = \limsup_{k \to \infty} D(u_k) =0.
\]
Consequently, $u$ is a Blaschke product by \Cref{la:zero-or-equality-defect-blaschke}.

Now assume \underline{$p \in (2,4)$}. Observe that $D(u)$ may not be well-defined for $u \in W^{s,\frac{1}{s}}$ when $s<\frac{1}{2}$, indeed this is related to \cite[Open problem 5.5]{BFavorite}. However, since $u$ is a minimizer, $u \in H^{\frac{1}{2}}(\S^1,\S^1)$. See \Cref{th:minH12}.
\footnote{Here the threshold $p=4$ is the blocking ingredient for studying the case $p>4$.}. Thus, again by smooth approximation, \Cref{la:continuityBuH12}, and \Cref{co:est} we conclude
\[
 D(u^2) = 4D(u)
\]
Thus we conclude by the equality case in \Cref{la:defect}.
\end{proof}

Lastly, \Cref{s:stability} follows from the following, combined with \Cref{co:est} and smooth approximation observing that we can apply \Cref{la:continuityBuH12} by Sobolev embedding.

\begin{corollary}[Local and global stability for $p<2$]
Let $p = \frac{1}{s} \in (1,2)$ and $u \in W^{s,\frac{1}{s}}(\S^1,\S^1)$.
\begin{itemize}
\item If $\deg u = 1$, then there exists a
constant $C_p>0$ such that
\[
\inf_{B_0\in{\rm B}_1^+}
\|u-B_0\|_{\dot H^{\frac12}(\S^1)}^2
\le
C\, D(u) .
\]
\item if $\deg u =d\geq 2$, we have the same inequality if in addition the assumption
\[
\inf_{B_0\in{\rm B}_d^+}\|u-B_0\|_{\dot H^{\frac12}(\S^1)}^2 < \eps_{d,p}
\]
for a small constant $\eps_{d,p}>0$
\end{itemize}
\end{corollary}
\begin{proof}
Since $p \in (1,2)$ this follows from \Cref{co:est} and \Cref{la:B-controls-Blaschke}.
\end{proof}

	\bibliographystyle{abbrv}%
	\bibliography{bib}%
\end{document}